\documentclass[12pt,oneside]{amsart}

\usepackage{graphicx}
\usepackage{amsfonts}
\usepackage{epsf}
\usepackage{amssymb}
\usepackage{amsmath}
\usepackage{amscd}
\usepackage{tikz}
\usepackage{pdfpages}
\usepackage{fancyhdr}
\usepackage{setspace}
\usepackage{hyperref}
\usepackage[all]{xy}
\usetikzlibrary{matrix}

\theoremstyle{definition}
\newtheorem{theorem}{Theorem}
\newtheorem{prop}[theorem]{Proposition}
\newtheorem{lem}[theorem]{Lemma}

\newtheorem{rem}[theorem]{Remark}

\newtheorem{example}[theorem]{Example}

\begin{document}

\rhead{\thepage}
\lhead{\author}
\thispagestyle{empty}


\raggedbottom
\pagenumbering{arabic}
\setcounter{section}{0}
\def\co{\colon\thinspace}

\title{Generic fibrations around multiple fibers}
\author{Kyle Larson}

\address{
Department of Mathematics \\ 
University of Texas at Austin   \\ 
Austin, TX\\
USA}
\email{klarson@math.utexas.edu}

\begin{abstract}
Given some type of fibration on a 4-manifold $X$ with a torus regular fiber $T$, we may produce a new 4-manifold $X_T$ by performing torus surgery
on $T$. There is a natural way to extend the fibration to $X_T$, but a multiple fiber (non-generic) singularity is introduced.
We construct explicit generic fibrations (with only indefinite fold singularities)
in a neighborhood of this multiple fiber. As an application this gives explicit constructions of broken Lefschetz
fibrations on all elliptic surfaces (e.g. the family $E(n)_{p,q}$). As part of the construction we produce generic fibrations around
exceptional fibers of Seifert fibered spaces.
\end{abstract}

\maketitle

\section{Introduction}

Various types of singular fibrations have proved to be powerful tools in the study of smooth 4-manifolds. The classical theories of Lefschetz and elliptic 
fibrations are a rich source of interesting examples and provide connections to algebraic geometry, symplectic geometry, and gauge theory.
More recently it has been shown that every smooth closed 4-manifold admits a broken Lefschetz fibration (see, for example, Akbulut and Karakurt \cite{AK}, Baykur \cite{Bay08} , Gay and Kirby \cite{GK}, and Lekili \cite{Lek}),
or alternatively a purely wrinkled fibration (which are also called indefinite Morse 2-functions).
On the other hand, torus surgery (also called a \emph{log transform}) is perhaps the most important surgical tool for 4-manifolds.
Indeed, by a result of Baykur and Sunukjian \cite{BS}, if $X'$ is an exotic (i.e. homeomorphic but not diffeomorphic) copy of the simply-connected and closed 4-manifold $X$, then $X'$ can be 
obtained from $X$ by some sequence of torus surgeries. 
In this paper we integrate these two perspectives by studying the result
of torus surgery on a regular fiber of a map to a surface. In particular, we construct
nice fibrations in a neighborhood of the glued in torus that agree with
the original fibration on the boundary. The existence of such fibrations
follows from a more general result of Gay and Kirby  \cite{GK}, but here we produce the first explicit
examples. Our work also fits nicely into the context of Baykur and Sunukjian \cite{BS}, where the
authors discuss when broken Lefschetz fibrations on different manifolds can be related by torus surgery and homotopy modifications of the fibration.
Our construction illustrates this for some specific examples.

\emph{Acknowledgments.} The author would like to thank the following people for helpful comments and conversations: \.{I}nan\c{c} Baykur, Stefan Behrens,
his advisor Robert Gompf, and \c{C}a\u{g}ri Karakurt.

\section{Torus surgery on a fiber}\label{torus}

Let $X$ be a smooth 4-manifold and $\Sigma$ a smooth surface, with $f \co X \to \Sigma$ some type of fibration map
(e.g. an elliptic fibration or broken Lefschetz fibration, but in general we just require $f$ to be proper and smooth). If $T \subset X$
is a regular fiber diffeomorphic to a torus, then we can identify a tubular neighborhood $\nu T$ with $T^2 \times D^2$ and a neighborhood of $f(T)$ with $D^2$ 
such that $\left.f\right|_{T^2 \times D^2}$ is projection onto the second factor. Let $\phi \co \nu T \to T^2 \times D^2$
be such an identification. Torus surgery on $T$ is the operation of cutting out $\nu T$ and
gluing in $T^2 \times D^2$ by $\phi ^{-1} \circ \psi$, where $\psi$ is a self-diffeomorphism of $\partial (T^2 \times D^2)$. 
Let $X_T$ be the resulting manifold $X \setminus \nu T \cup _{\phi ^{-1} \circ \psi} T^2 \times D^2$. Since gluing in $T^2 \times D^2$
amounts to attaching a 2-handle, two 3-handles, and a 4-handle, the diffeomorphism type of $X_T$ is determined by the attaching sphere of the 2-handle:
$\phi ^{-1} \circ \psi (\{pt\} \times \partial D^2)$ (the framing is canonical). The isotopy class of this curve is then determined
by the homology class $\gamma = \psi _\ast [\{pt\} \times \partial D^2] \in H_1 (T^2) \oplus \mathbb{Z}$, where the $\mathbb{Z}$ factor is
generated by $m = [\{pt\} \times \partial D^2]$. Now $\gamma$ must be a primitive element, so $\gamma = q\alpha + pm$ for relatively prime
integers $p$ and $q$ and $\alpha$ a primitive element of $H_1 (T^2)$.
Hence, given our identification $\phi$, $X_T$ is determined up to diffeomorphism by the data $p$, $q$, and $\alpha$, which are called
the \emph{multiplicity}, the \emph{auxiliary multiplicity}, and the \emph{direction}. We say the surgery is \emph{integral} if $q= \pm 1$.
For more exposition see Gompf and Stipsicz \cite{GS}.

Now fixing a specific torus surgery determined by $p$, $q$, and $\alpha$, we may change our identification $\phi$ so that the
direction $\alpha$ corresponds to the second $S^1$ factor of $T^2 \times D^2 = S^1 \times S^1 \times D^2$. To be precise, we
compose $\phi$ with a map $g \times id \co T^2 \times D^2 \to T^2 \times D^2$, where $g$ is some self-diffeomorphism of
$T^2$ that sends a curve representing $\alpha$ to $\{pt\} \times S^1$. We abuse notation by renaming this new identification $\phi$.
In doing this we have not changed the surgery,
but we have changed how we look at a neighborhood of $T$ in order to make things more convenient for what follows.

We are interested in which surgeries on $T$ allow the fibration $\left.f\right|_{X \setminus \nu T}$ to be extended over $X_T$.
By our above remarks, up to diffeomorphism we can choose our gluing map $\psi$ to be (thinking of $\partial (T^2 \times D^2)$ as
$\mathbb{R}^3 / \mathbb{Z}^3$):
\[\psi = \left( \begin{array}{ccc}
1 & 0 & 0 \\
0 & (qk+1)/p & q \\
0 & k & p \end{array} \right)\] 
where $k$ is an integer satisfying $qk+1 \equiv 0$ mod $p$ (if $p=0$, set the center entry to 0). If we instead think of $T^2 \times D^2$ as
$\{ (\xi_1, \xi_2, z) \subset \mathbb{C}^3 \mid \xi_i \in S^1 \subset \mathbb{C}, z \in D^2 \subset\mathbb{C}\}$,
then we can write $\psi$ multiplicatively as $\psi (\xi_1, \xi_2, z) = (\xi_1, \xi_2^{(qk+1)/p} \cdot z^q, \xi_2^k \cdot z^p)$ (see Harer, Kas, Kirby \cite{HKK} for more information).
Now we can see that if $p \ne 0$ the fibration extends over the glued in $T^2 \times D^2$ by 
defining $\hat f \co X_T \to \Sigma$ by
\begin{equation*}
\hat f(x) = \left\{
	\begin{array}{ll}
		f(x)  & \mbox{if } x \in X \setminus \nu T \\
		\xi_2^k \cdot z^p & \mbox{if } x = (\xi_1, \xi_2, z) \in T^2 \times D^2
	\end{array}
\right.
\end{equation*}
One can check that the fibration is exactly $S^1$ times the fibration around a $(p, -k)$ exceptional fiber in a Seifert fibered space.
Hence the central fiber $T=T^2 \times \{0\}$ is $p$-times covered by nearby fibers and the homology class of a nearby fiber $[F] = p \cdot [T]$.
Furthermore, if $p > 1$ then
one can compute in local coordinates that $d \hat f$ vanishes on $T$ and is a submersion everywhere else in $T^2 \times D^2$ 
(if $p = 1$ the fibration extends over $T^2 \times D^2$ with no singularity). For $p > 1$
we say that $T$ is a \emph{multiple fiber} singularity of $\hat f$. Since $d \hat f$ vanishes on a 2-dimensional subspace, $\hat f$
cannot be a generic map to a surface (near $T$).

The purpose of this paper is to construct indefinite generic fibrations on $T^2 \times D^2$
that agree with $\hat f$ on $\partial (T^2 \times D^2)$.

\section{Constructing generic fibrations}

Our strategy will be to construct generic fibrations using round handles. An $(n+1)$-dimensional round $k$-handle is $S^1 \times h_k^n$,
where $h_k^n$ is an $n$-dimensional $k$-handle, and it is attached along $S^1$ times the attaching region of $h_k^n$ (see Baykur \cite{Bay09}, Baykur and Sunukjian \cite{BS} for more information about round handles).
If we are attaching a round handle to a manifold whose boundary fibers over $S^1$, so that a single $h_k^n$ is attached to each fiber,
then we can extend the boundary fibration over the round $k$-handle by taking the Morse level sets of each $h_k^n$
(and adjusting the fibration in a collar neighborhood of the boundary). However, this
fibration will have a \emph{fold singularity}, which by definition is a singular set that locally looks like $\mathbb{R}$ times a Morse singularity.
More precisely, there exist local coordinates $(t, x_1, \cdots, x_n)$ around each critical point such that the fibration map is given by 
$(t, x_1, \cdots, x_n) \mapsto (t, x_1^2 \pm \cdots \pm x_n^2)$ in these coordinates. Importantly for our purposes, fold singularities
of maps to surfaces are a generic type of singularity.
The fold singularity is called \emph{indefinite} if the Morse singularity in the above coordinates is indefinite
(i.e. the Morse critical point does not have index equal to 0 or $n$). 

\begin{rem}
 
In what follows we will abuse terminology and call a fibration \emph{generic} if its singularities consist of only
indefinite fold singularities. It is a fact from singularity theory that such fibrations do belong to the set of generic (and stable) maps, but these
fibrations are actually a very special subset of all generic maps. Indeed, maps from a 4-manifold to a surface with only these singularities
are a subset of both broken Lefschetz fibrations (which can also contain Lefschetz singularities) and purely wrinkled fibrations (which can also
contain indefinite cusp singularities).

\end{rem}

\begin{figure}
\includegraphics[scale=0.4]{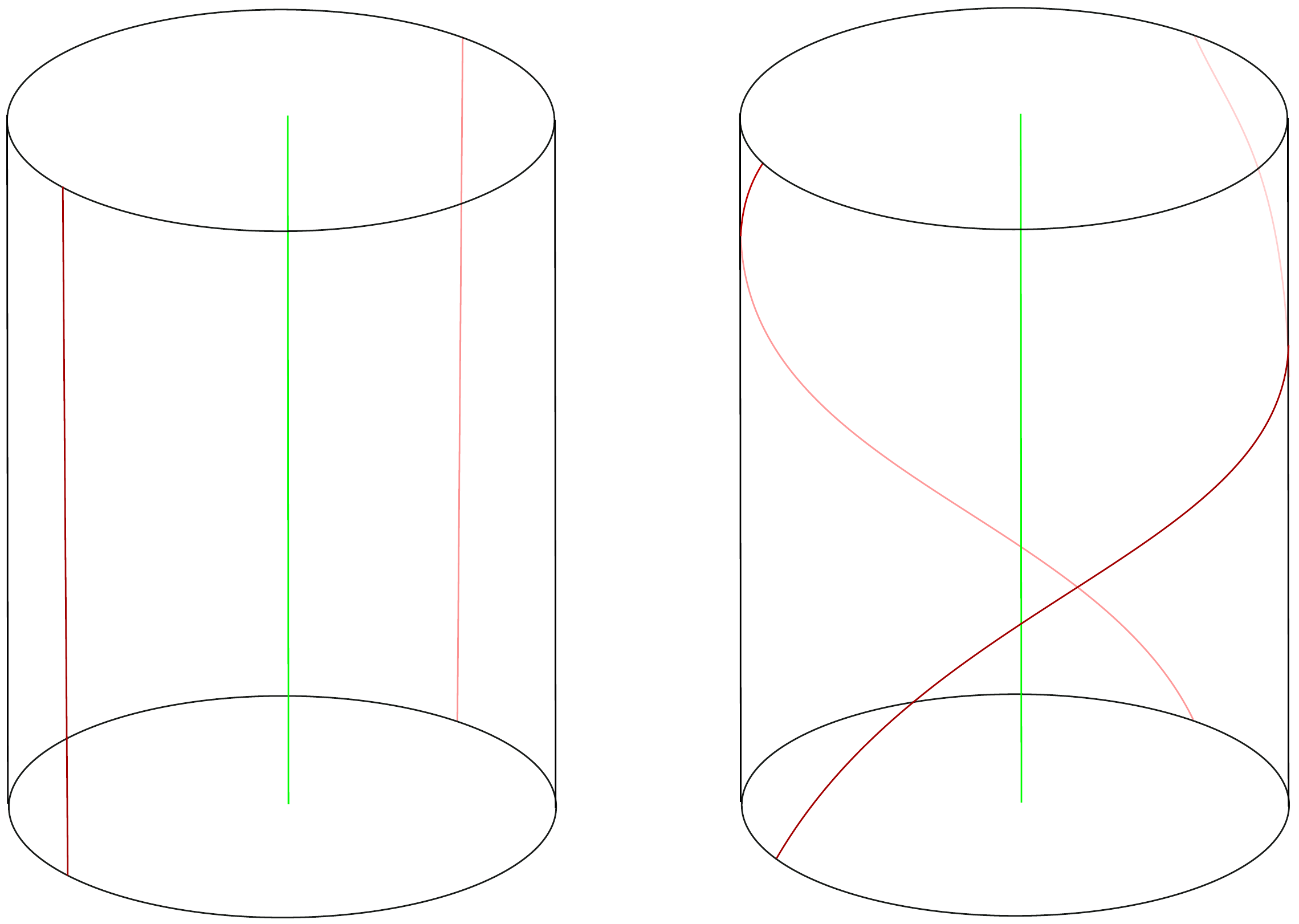}
\centering
 \caption{A neighborhood of a (2,1) exceptional fiber. On the left the ends of the solid cylinder are identified by a 180 degree twist, whereas
 on the right the ends are identified by the identity map. The green arc becomes the exceptional fiber under the identification, and the red arcs
 become a single fiber on the boundary.}
\label{exfibers}
\end{figure}

First we do our construction for a neighborhood of an exceptional fiber
in a Seifert fibered space. Recall that the neighborhood of a $(p,q)$ exceptional fiber can be formed by taking a solid cylinder and identifying the two ends with
a $2\pi q/p$ twist. In fact, we start with the simplest possible case: the neighborhood of a (2,1) exceptional fiber.
Let $N$ be a tubular neighborhood of a (2,1) exceptional fiber. Then $N$ is diffeomorphic to a solid cylinder with ends identified
with a 180 degree twist (see Figure \ref{exfibers}). The exceptional fiber is the circle formed by identifying the two ends of the central arc. A regular fiber of $N$
consists of two arcs opposite each other and equidistant from the central arc, which form a single circle
after the identification of the ends of the cylinder. We can also view $N$
as a solid cylinder with ends identified by the identity map, but now regular fibers twist around the central fiber (see the second picture of Figure \ref{exfibers}).
Let $f \co N \to D^2$ be the fibration map (note that $f$ is not simply the projection of the solid cylinder; we have to compose
with the  2 to 1 branched covering map of the disk).

 \begin{figure}
\includegraphics[scale=0.4]{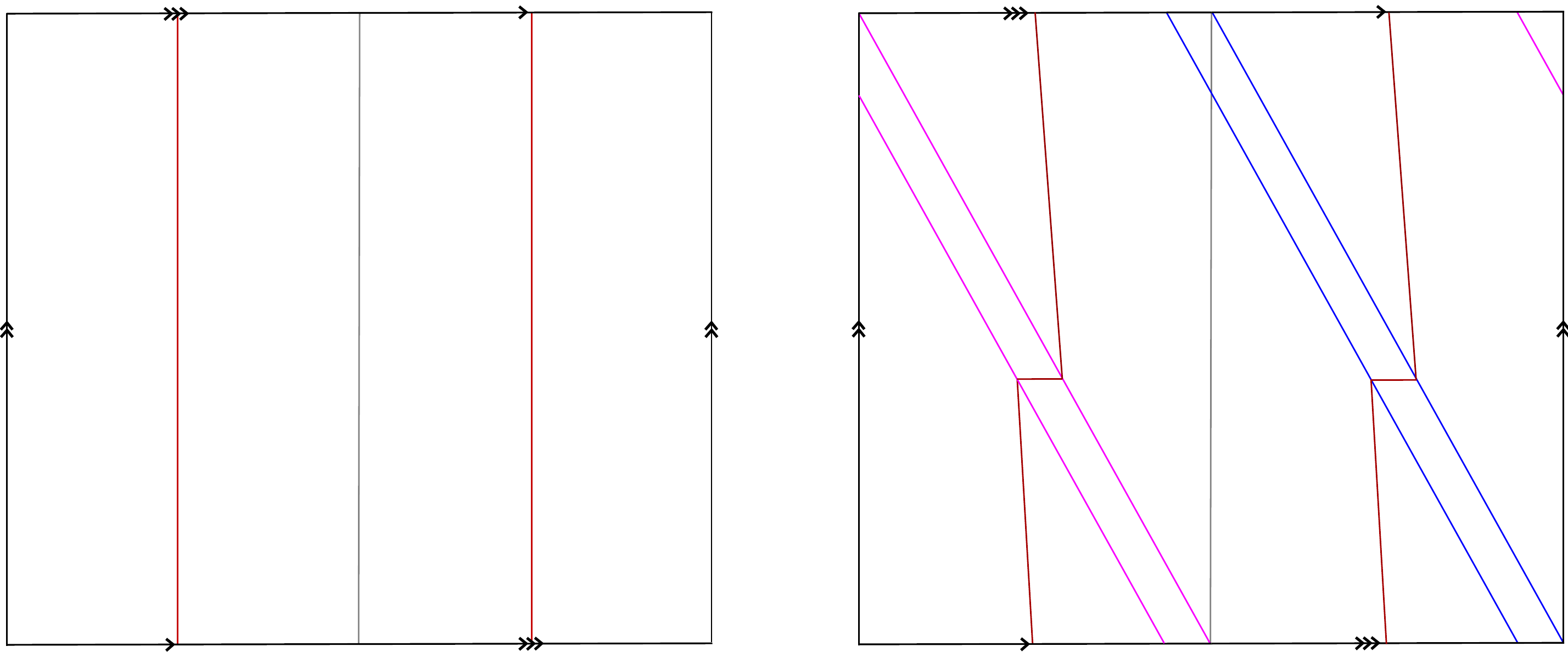}
 \centering
 \caption{The fibering on $\partial N$. The top edges are identified with the bottom edges with a 1/2 shift to the right.
 The red arcs form a single fiber (note that the gray middle arc is \emph{not} part of the fibration),
 and we see the result of the isotopy of the fibration on $\partial N$ in the second picture.
 The diagonal strips form the attaching region for the 3-dimensional round 1-handle, and the two horizontal sections of the red fiber
 are the attaching regions for the 2-dimensional 1-handle.}
\label{squares}
\end{figure}
 
 \begin{figure}
\includegraphics[scale=0.4]{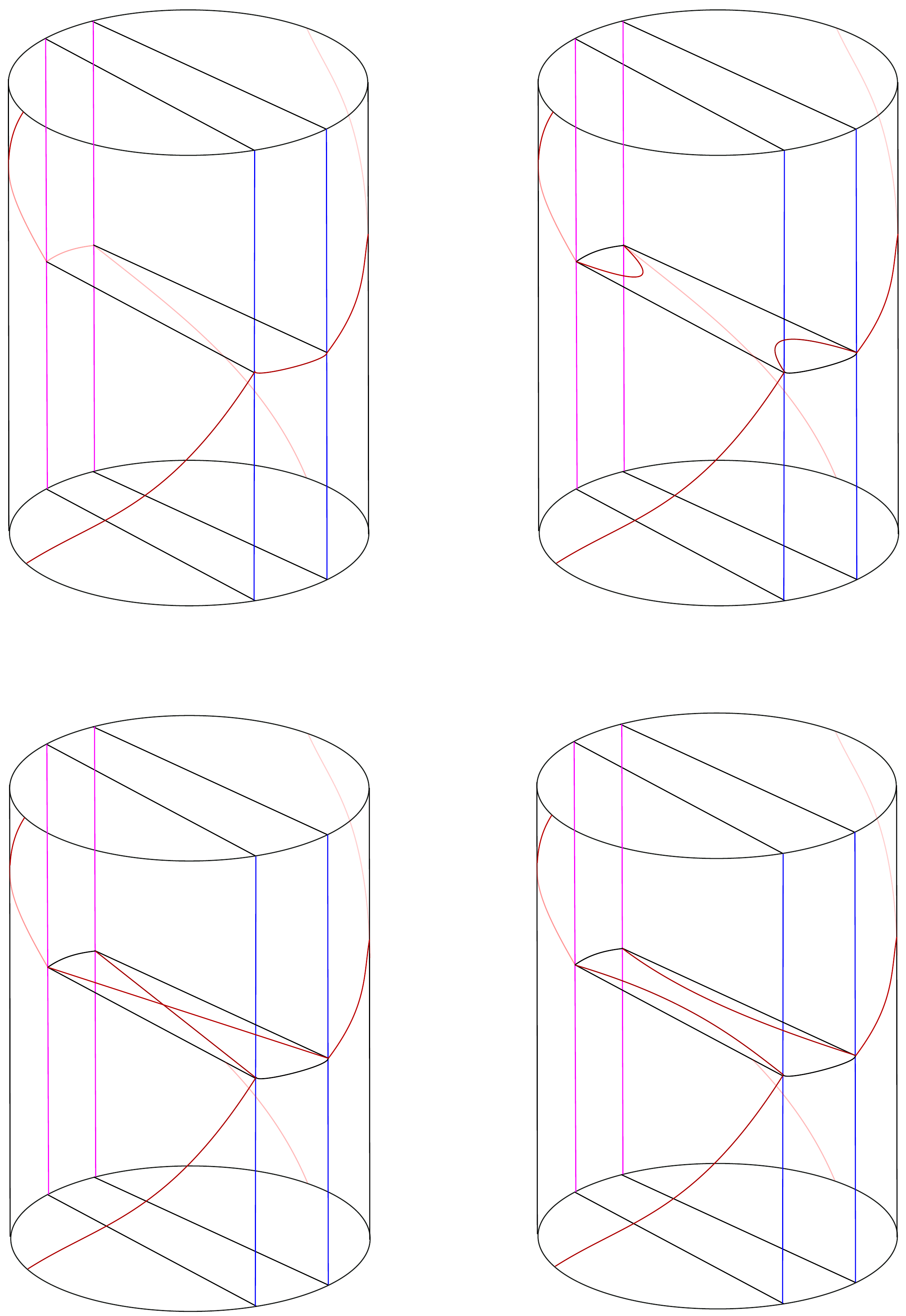}
 \centering
 \caption{Extending the fibration across $R$. The top and bottom of the cylinders are identified by the identity map,
 and we see $R$ as the rectangular prism with top and bottom identified. In each picture the red arcs
 form a single fiber. The first picture shows the fibration on $\partial N$ (after the isotopy),
 and in the following pictures the fiber get pushed across the
 2-dimensional 1-handle (while the arcs on the boundary of the cylinder actually live in a collar $\partial N \times I$).
 In the first two pictures the fiber is a single circle wrapping twice around the cylinder.
 The third picture is the singular level, where the fiber consists of the wedge of two circles. The last picture shows a fiber past
 the singular level, and the fiber consists of two disjoint circles. Here we see that after extending the fibration across $R$ we get two
 ``chambers''  with torus boundaries, each fibered by (1,1) curves.}
\label{cylinders}
\end{figure}

\begin{lem} 
$N$ admits a generic fibration $\hat f \co N \to D^2$ such that $\left.\hat f\right|_{\partial N} = \left.f\right|_{\partial N}$,
with one indefinite fold singular locus. The image of the critical set is an embedded circle in $D^2$, and the preimage of a point in the
interior of this circle is two disjoint circles.
\end{lem}

\begin{proof}
 
 Our strategy will be to delete int$N$, and then fill it back in (relative to the boundary) 
 with one round 1-handle and two trivially fibered solid tori in
 such a way as to extend the fibration on $\partial N$. Recall a 3-dimensional round 1-handle is a copy of $S^1 \times D^1 \times D^1$
 attached along an embedding of $S^1 \times \partial D^1 \times D^1$. We can think of this as adding a circle's worth of 2-dimensional
 1-handles. In our case we attach a single 2-dimensional 1-handle to each $S^1$ fiber of $\partial N$.
 Now there are two ways to attach a 2-dimensional
 1-handle to $S^1$, resulting in either one or two components (depending on whether the 1-handle preserves or reverses orientation).
 We will attach the round 1-handle so that the resulting
 fibers have two components. Before we attach the round 1-handle we modify the fibration on $\partial N$ by an isotopy.
 
 If we restrict our attention to the fibration on $\partial N$, thought of as the boundary of the cylinder with top and bottom identified
 with a 180 degree twist, then we can cut the cylinder open and think of $\partial N$ as a square with left and right edges identified by the
 identity map and the top edge identified to the bottom edge by a 1/2 unit shift to the right (see Figure \ref{squares}). We see the fibers of
 $\partial N$ as a pair of vertical arcs separated by 1/2 units in the horizontal direction. Our modification of the fibration on $\partial N$ involves
 isotoping the fibers (in a collar $\partial N \times I$) so that each fiber is horizontal along the two diagonal strips in the second picture
 of Figure \ref{squares}. The diagonal strips will form the attaching region of the round 1-handle, and the two horizontal sections of each fiber will be the
 attaching region of the 2-dimensional 1-handle to each fiber.
 
 Now we can extend the fibration on $\partial N$ across the round 1-handle (which we will denote by $R$) as follows (see Figure \ref{cylinders}):
 $\partial N$ consists of a circle's worth of $S^1$ fibers, and attaching $R$ has the effect of attaching a 2-dimensional 1-handle to each fiber.
 We extend the fibration over each of these 1-handles by taking the level sets corresponding to the natural Morse function on
 1-handle $\cup$ (fiber $\times I$), where the $I$ factor comes from a collar neighborhood $\partial N \times I$.
 So before the critical level the fibers will be circles, the critical level will be the wedge of two circles,
 and after the critical level the fibers will be a disjoint union of two circles. Therefore, adding the round 1-handle $R$
 introduces a fold singularity $C$ ($S^1 \times$ the Morse critical point of the 2-dimensional 1-handle), and 
 $\hat f$ maps $C$ to an embedded circle. The boundary $\partial(\partial N \cup R)$ is
 two disjoint tori (here we are only considering the ``interior'' part of the boundary, the exterior of course consists of another torus).
 Furthermore, as we can see in Figure \ref{cylinders}, each of these tori are fibered
 with multiplicity 1 (i.e. fibered by (1,1) curves). So we see that adding the round 1-handle reduces the 
 multiplicity from 2 to 1 at the expense of increasing the number of components of a fiber from 1 to 2. Now we can fill in these two tori
 with two trivially fibered solid tori (the (1,1) fibration on the boundary extends over the solid torus without singularities).
 Topologically we are just gluing back in the two solid tori of $N \setminus (\partial N \cup R)$, but in such a way as to extend the fibration.
 This completes our construction of a generic fibration on $N$.
 
\end{proof}

It is quite easy to extend our construction to the case of a $(p, 1)$ exceptional fiber:

\begin{prop}
If $N$ is a tubular neighborhood of a (p,1) exceptional fiber and $f \co N \to D^2$ is the fibration map, then
$N$ admits a generic fibration $\hat f \co N \to D^2$ constructed with $p-1$ round handles
such that $\left.\hat f\right|_{\partial N} = \left.f\right|_{\partial N}$.
The image of the critical set is $p-1$ embedded circles in $D^2$, and the preimage of a point in the
interior of this circle is $p$ disjoint circles.
\end{prop}

\begin{proof}
 
 We proceed as before, by starting with the fibration on $\partial N$, and attaching a 3-dimensional round 1-handle along the two strips as 
 in Figure \ref{fibrationsquare3} (after isotoping the fibration in a collar $\partial N \times I$ so that fibers are horizontal across the diagonal strips).
 We extend the fibration across the round handle as before, and the resulting interior boundary is again two tori, but this time one has
 multiplicity 1 and the other has multiplicity $p-1$. The torus fibered with multiplicity 1 can be filled with a trivially fibered solid torus,
 and we repeat this procedure inductively with the torus fibered with multiplicity $p-1$. The result is that we consecutively attach $p-1$
 round 1-handles (each one increasing the number of components of a fiber by 1) and glue in $p-1$ trivially fibered solid tori. This gives
 the required generic fibration on $N$.
 
\end{proof}

 \begin{figure}
\includegraphics[scale=0.4]{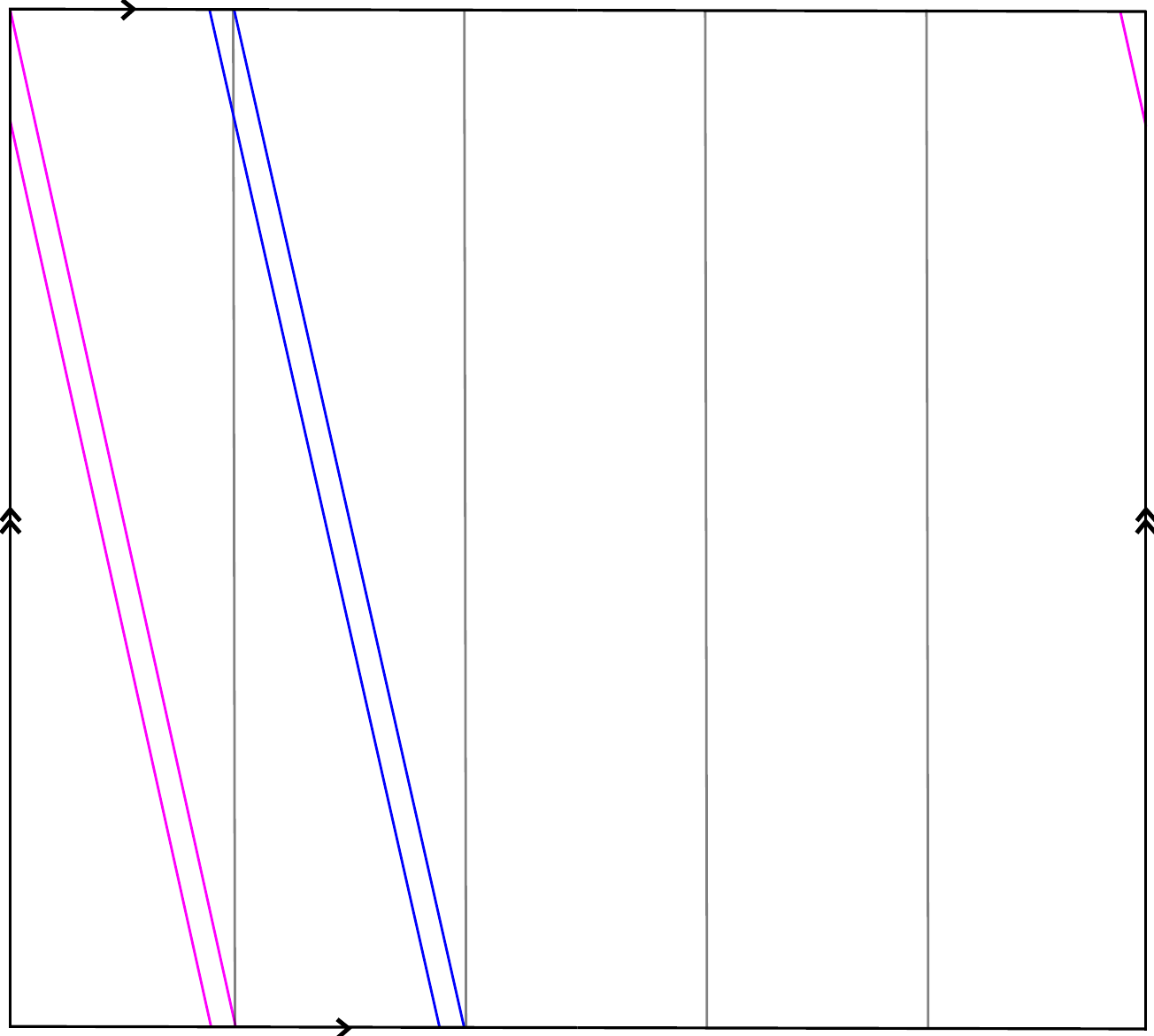}
 \centering
 \caption{Attaching a 3-dimensional round 1-handle to $\partial N$. 
 The diagonal strips form the attaching region for the first round 1-handle.
 Here we draw the case for a (5,1) exceptional fiber, 
 but the picture obviously generalizes to a $(p,1)$ exceptional fiber.}
\label{fibrationsquare3}
\end{figure}

One can construct generic fibrations in a neighborhood of a $(p,q)$ exceptional fiber in a similar manner, but the author has not worked out
a general algorithm.

We now proceed to the construction of generic fibrations around a torus multiple fiber. Here we describe the process for singular fibrations
resulting from \emph{integral} surgeries (which will suffice for our applications), but again, one could apply these techniques to non-integral
surgeries as well.

\begin{theorem}
The fibration around a multiple fiber singularity resulting from an integral torus surgery of multiplicity $p$ can be replaced with a generic
fibration (extending the fibration on the boundary) composed of $(p-1)$ round 1-handles and $(p-1)$ round 2-handles. The image of 
the indefinite fold critical set is $2 \cdot (p-1)$ consecutively embedded circles and the preimage of an interior point consists of $p$ disjoint tori.
\end{theorem}

\begin{proof}
 
As before, we will start with the multiplicity 2 case and then generalize to multiplicity $p$. Let $M$ be the neighborhood of the multiple fiber
singularity. By our remarks  in Section \ref{torus}, $M$ is fiber-preserving diffeomorphic to $S^1 \times N$, where $N$ is the fibered neighborhood of a (2,1)
exceptional fiber. Let us assume that the fibration on $\partial M = S^1 \times \partial N$ has been modified by isotopy so that the fibration is 
$S^1$ times the modified fibration on $\partial N$.
We will use our generic fibration on $N$ to construct a generic fibration on $M$, however,
the fibration is \emph{not} just $S^1$ times the generic fibration on $N$. In that case the singular set would be $S^1 \times C$, where $C$
is the singular circle of the generic fibration on $N$, and 2-dimensional singular sets do not occur generically.
In what follows it will be helpful to refer to Figure \ref{cylinders} and think of $M = S^1 \times N$ as a ``movie'' where time is the $S^1$ direction.

The generic fibration on $N$ was constructed using a 3-dimensional round 1-handle $R$, which we thought of as a circle's worth of
2-dimensional 1-handles. We use this family of 2-dimensional 1-handles to construct a 4-dimensional round 1-handle $R_1^4$ and a
4-dimensional round 2-handle $R_2^4$ as follows: Let $\theta$ parametrize the $S^1$ factor of $R = S^1 \times D^1 \times D^1$
(hence $\theta$ parametrizes the family of 2-dimensional 1-handles)
and let $x$ and $t$ be coordinates on the two $D^1$ factors. Define two subsets $I_1, I_2 \subset S^1$ by
$I_1 = \{(\cos \theta, \sin \theta) \in \mathbb{R}^2 \mid \theta \in [-\pi/4, \pi/4]\}$ and $I_2 = S^1 \setminus I_1$.
Let $g \co R \hookrightarrow N$ denote the embedding map from our previous construction (note this is not simply the attaching map,
but in fact embeds the entire round handle into $N$), and 
let $R_1^4 = S^1 \times I_1 \times D^1 \times D^1$ where $\varphi$ parametrizes the $S^1$ factor.

Embed $R_1^4$ into $M$ by the map
$G_1 \co S^1 \times I_1 \times D^1 \times D^1 \hookrightarrow S^1 \times N$, $G_1 (\varphi, \theta, x, t) = (\varphi + \theta, g(\varphi, x, t))$
(It is important to note that $g$ now takes $\varphi$ as input instead of $\theta$). We can think of attaching $R_1^4$ to 
$\partial M = S^1 \times \partial N$ by $\left.G_1\right|_{S^1 \times I_1 \times \partial D^1 \times D^1}$.
For a fixed value of $\varphi$, say $\varphi_0$, $\left.g\right|_{\{\varphi_0\} \times \partial D^1 \times D^1}$ 
maps to a single circle fiber $c$ of $\partial N$,
and so $\left.G_1\right|_{\{\varphi_0\} \times I_1 \times \partial D^1 \times D^1}$ is an attaching map for a 3-dimensional
1-handle to the torus fiber $S^1 \times c \subset S^1 \times \partial N = \partial M$. Indeed we see that 
$\left.G_1\right|_{\{\varphi_0\} \times I_1 \times D^1 \times D^1}$ embeds a 3-dimensional 1-handle into $M$  by embedding the 
2-dimensional 1-handle ``slices'' $\{\theta\} \times D^1 \times D^1$ into $\{\varphi_0 + \theta \} \times N \subset M$ for $\theta \in I_1$
(see Figure \ref{torusfiber}).
Letting $\varphi$ range over $S^1$ shows that attaching $R_1^4$ amounts to adding a 3-dimensional 1-handle to each torus fiber of
$S^1 \times \partial N$, so that the genus of the fibers increases by one.
However, this is done in a way such that if we look at the result of attaching $R_1^4$ in a single frame of our ``movie,''
$(\partial M \cup R_1^4) \cap (\{pt\} \times N)$, we see a 3-dimensional 1-handle attached to $\partial N$, but which is composed of
2-dimensional 1-handle ``slices,'' each slice belonging to a different $\left.G_1\right|_{\{\varphi\} \times I_1 \times D^1 \times D^1}$.
The result is that the fold singular set of $R_1^4$ intersects a frame of our movie, $\{pt\} \times N$, in a single point,
corresponding to the Morse singularity of the 3-dimensional 1-handle $\left.G_1\right|_{\{pt\} \times I_1 \times D^1 \times D^1}$
(which occurs at $\theta = 0 \in I_1$).

We embed the 4-dimensional round 2-handle $R_2^4 = S^1 \times I_2 \times D^1 \times D^1$ into $M$ similarly, by the map
$G_2 (\varphi, \theta, x, t) = (\varphi + \theta, g(\varphi, x, t))$ (indeed this is the same map, except $\theta$ now takes values in $I_2$).
We can think of attaching $R_2^4$ to $\partial (\partial M \cup R_1^4)$ by $\left.G_2\right|_{S^1 \times \partial (I_2 \times D^1) \times D^1}$,
and one can check that this amounts to attaching a 3-dimensional 2-handle $\left.G_2\right|_{\{\varphi\} \times I_2 \times D^1 \times D^1}$
to each genus 2 fiber of $\partial (\partial M \cup R_1^4)$. Now $\left.G_2\right|_{\{\varphi\} \times I_2 \times D^1 \times D^1}$ is 
actually a separating 3-dimensional 2-handle, and we can see this by looking at Figure \ref{torusfiber} (here we again consider a fixed $\varphi = \varphi_0$).
As $\theta$ varies over $I_2$ we add more 2-dimensional 1-handle slices to the picture whose attaching regions fill out the remainder of the 
two annuli. We see that this amounts
to attaching a 3-dimensional 2-handle to the genus 2 fiber whose attaching circle runs twice over the 1-handle. From the picture we see that
this is a separating 2-handle that results in two disjoint torus components.

Another way to see this is by considering the resulting fibration on the boundary: if we look at a frame of our movie after attaching
$R_1^4$ and $R_2^4$, $(\partial M \cup R_1^4 \cup R_2^4) \cap (\{pt\} \times N)$, we see $(\{pt\} \times \partial N) \cup R$, but the 
2-dimensional slices of $R$ in this frame belong to different slices of $R_1^4$ and $R_2^4$. The point is that topologically $\partial M \cup R_1^4 \cup R_2^4$
gives a decomposition of $S^1 \times (\partial N \cup R)$, but in such a way that the natural fibrations on the boundaries agree.
That is, the fibers of $\partial(\partial M \cup R_1^4 \cup R_2^4)$ are exactly $S^1$ times the fibers of $\partial (\partial N \cup R)$.
This means that after attaching $R_1^4$ and $R_2^4$ the fibers consist of two disjoint tori.
Furthermore, we observe that since the fibration on $N$ was completed by adding two trivially fibered solid tori to $\partial (\partial N \cup R)$,
our fibration on $M$ is completed by adding two trivially fibered $T^2 \times D^2$'s to 
$\partial M \cup R_1^4 \cup R_2^4 = S^1 \times (\partial N \cup R)$ (again we have reduced the multiplicity from 2 to 1).

To go from the multiplicity 2 case to the general case of multiplicity $p$, we repeat the above construction inductively
using the generic fibration around a $(p,1)$ exceptional fiber. The result will be a generic fibration with $p-1$ pairs of
4-dimensional round 1- and 2-handles added in succession. Each round 1-handle raises the genus by one on a single component
of a fiber, and then the following round 2-handle splits the genus 2 component into two tori. Therefore the preimage of a point
in the interior of the round singular images will be the disjoint union of $p$ tori.

\end{proof}

\begin{figure}
\includegraphics[scale=0.5]{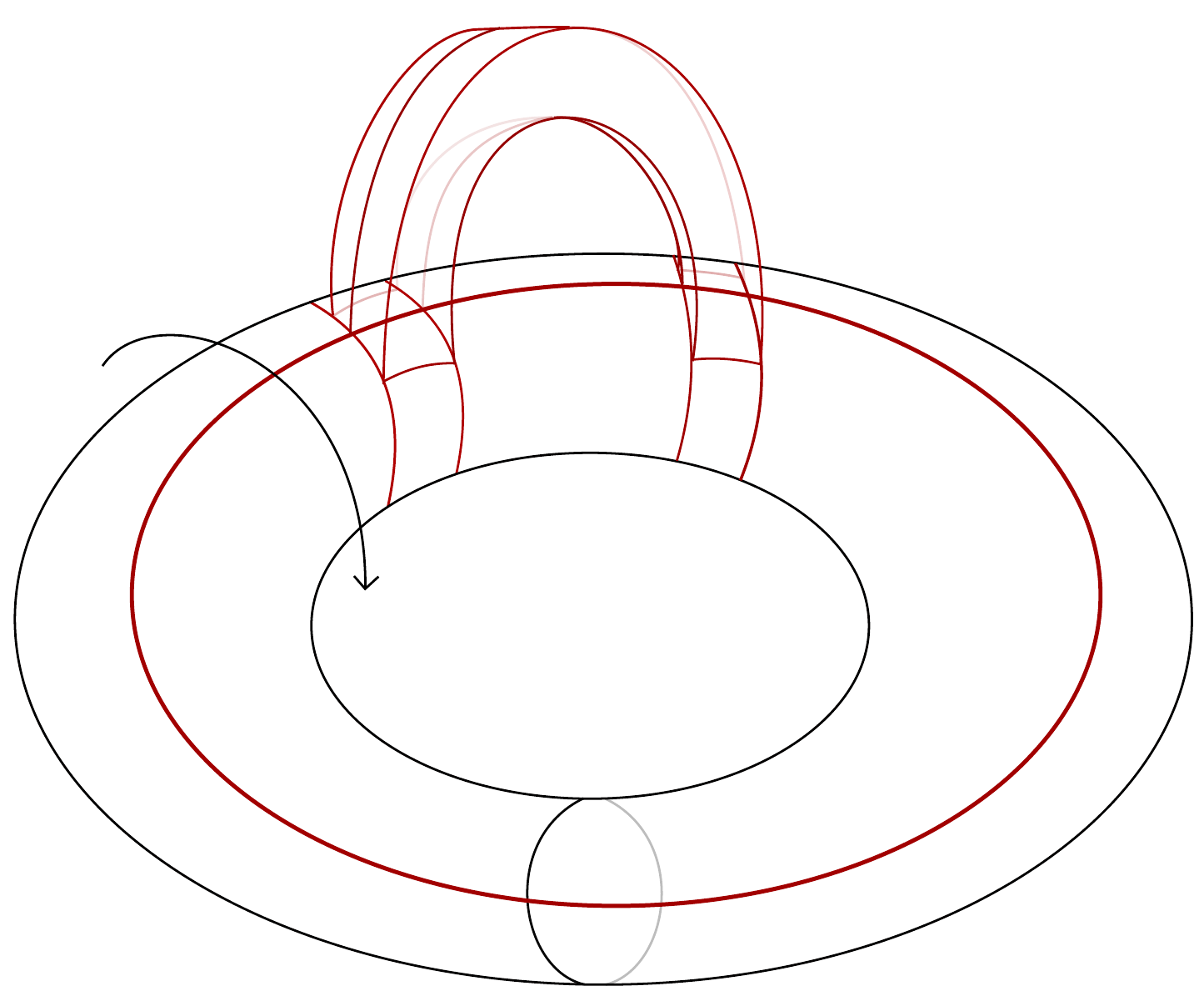}
 \centering
 \caption{Here we change our perspective and consider the contribution
 of $R_1^4$ to a single torus fiber $S^1 \times c \subset S^1 \times \partial N$ (the arrow in the picture shows the $S^1$ direction).
 The red circle on the torus is a single $\{pt\} \times c$, and adding $R_1^4$ corresponds to attaching a 2-dimensional 1-handle to
 the circle in a single frame $\{pt\} \times \partial N$. As $\theta \in I_1$ varies we add an interval's worth of 2-dimensional 1-handles
 that fill out a 3-dimensional 1-handle attached to our torus fiber.}
\label{torusfiber}
\end{figure}

Now that we have constructed one such generic fibration around a multiple fiber singularity, it is easy to produce others using the
homotopy moves of Baykur \cite{Bay08}, Lekili \cite{Lek}, and Williams \cite{Wil}.

\section{An application}

We conclude this note by applying our construction to give explicit broken Lefschetz fibrations (BLFs) on an important family of 4-manifolds:
the \emph{elliptic surfaces}. Our construction may be helpful for studying exotic behavior from the point of view of BLFs. 

An elliptic surface is a 4-manifold that admits a (possibly singular) fibration over a surface such that a regular fiber is diffeomorphic
to a torus, and with the extra condition that the fibration is locally holomorphic. Up to diffeomorphism we can assume that an elliptic surface
is equipped with a fibration map with only Lefschetz singularities and multiple fiber singularities coming from torus surgery
(see, for example, Gompf and Stipsicz \cite{GS}). Hence we can use
our construction to replace the fibration around a multiple fiber with a fibration with only indefinite fold singularities. 
The resulting fibration is by definition a BLF (since the only other singularities are Lefschetz singularities).
Of particular interest are the families of simply-connected exotic elliptic surfaces $E(n)_{p,q}$. 
The notation means that torus surgery is performed on two separate regular fibers of $E(n)$, one with multiplicity $p$ and the other
with multiplicity $q$ for relatively prime $p$ and $q$.
These regular fibers will lie in a cusp neighborhood, and so up to diffeomorphism the surgeries are determined by their multiplicity.
Hence we can apply our construction using \emph{integral} multiplicity $p$ and multiplicity $q$ surgery.

Lastly, we consider the special case of the Dolgachev surface $E(1)_{2,3}$.

\begin{example}
We construct a BLF on $E(1)_{2,3}$ by replacing the torus multiple fiber of multiplicity 2 with a generic fibration with two fold singularities
coming from a round 1-handle and a round 2-handle. The torus multiple fiber of multiplicity 3 is replaced with a generic fibration with 4 fold
singularities coming from two successive pairs of round 1- and 2-handles.
In Figure \ref{E_1_2,3} we draw the critical image on the base $S^2$ for the BLF on $E(1)_{2,3}$ resulting from our construction.
\end{example}

\begin{figure}
\includegraphics[scale=0.5]{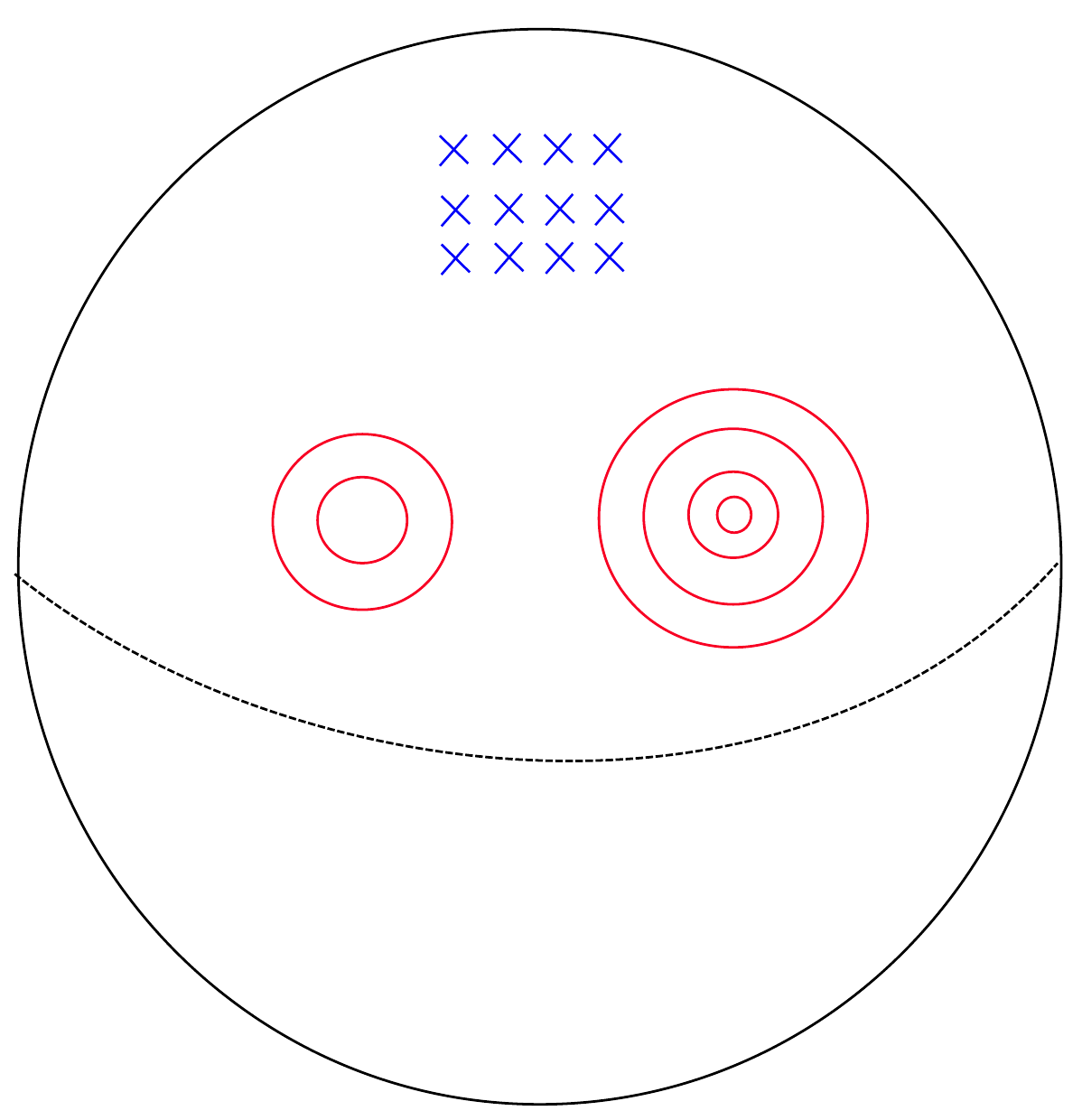}
 \centering
 \caption{The critical image on the base $S^2$ for a BLF on $E(1)_{2,3}$. The blue x's are the images of the 12 Lefschetz critical points,
 and the red circles are the images of the fold singularities.}
\label{E_1_2,3}
\end{figure}

%
%
%
\bibliographystyle{gtart}
\bibliography{multiplefibers.bib}



\end{document}